\documentclass[11pt]{amsart}
\usepackage{amsmath, amssymb, amsthm,verbatim,xypic,epic,eepic,longtable}

\numberwithin{equation}{section}

\newtheorem{thm}[equation]{Theorem}
\newtheorem{prop}[equation]{Proposition}
\newtheorem{lemma}[equation]{Lemma}
\newtheorem{cor}[equation]{Corollary}

\theoremstyle{remark}
\newtheorem*{remark}{Remark}
\newtheorem{defn}[equation]{Definition}

\newcommand{\F}{\mathbb{F}}

\begin{document}

\title[Planar functions and perfect nonlinear monomials]{Planar functions and perfect nonlinear monomials over finite fields}

\author{Michael E. Zieve}
\address{
  Department of Mathematics,
  University of Michigan,
  530 Church Street,
  Ann Arbor, MI 48109-1043 USA
}
\email{zieve@umich.edu}
\urladdr{http://www.math.lsa.umich.edu/$\sim$zieve/}

\date{January 21, 2013}

\thanks{The author was partially supported by NSF grant DMS-1162181.}
 

\begin{abstract}
The study of finite projective planes involves planar functions, namely,
functions $f\colon\F_q\to\F_q$ such that, for each $a\in\F_q^*$, the function $c\mapsto f(c+a)-f(c)$
 is a bijection on $\F_q$.  Planar functions are also used in the construction of DES-like cryptosystems, where they are called perfect nonlinear functions.  
We determine all planar functions on $\F_q$ of the form $c\mapsto c^t$,
under the assumption that $q\ge (t-1)^4$.
This implies two recent conjectures of Hernando, McGuire and Monserrat.
Our arguments also yield a new proof of a conjecture of Segre and Bartocci from 1971
concerning monomial hyperovals in finite Desarguesian projective planes.
\end{abstract}

\maketitle


\section{Introduction}
Let $q=p^r$ where $p$ is prime and $r$ is a positive integer.
A \textit{planar function} is a function
$f\colon\F_q\to\F_q$ such that, for every $a\in\F_q^*$, the function $c\mapsto f(c+a)-f(c)$ is a bijection on $\F_q$.
Planar functions can be used to construct finite projective planes, and they have been studied by finite geometers since 1968 \cite{DO}.
They arose more recently in the cryptography literature where they are called \textit{perfect nonlinear functions} \cite{NK}, the idea being that these functions are optimally resistant to linear and differential cryptanalysis when used in DES-like cryptosystems.
Many authors have investigated the planarity of monomial functions $f(x)=x^t$ with $t>0$.
Since $x^t$ is planar on $\F_q$ if and only if $x^{t+q-1}$ is planar, and likewise if and only if $x^{tp}$ is planar, the study of planar monomials reduces at once to the case that $t<q$ and $p\nmid t$.

\stepcounter{equation}
\stepcounter{equation}
\stepcounter{equation}
The only known examples of planar monomials $x^t$ over $\F_{p^r}$
with $p\nmid t$ and $t<p^r$ are
\begin{enumerate}
\item[(1.1)] $t=p^i+1$ if $0\le i<r$ and $p\frac{r}{\gcd(i,r)}$ is odd; \,and
\item[(1.2)] $t=\frac{3^i+1}2$ if $p=3$ and $2<i<r$ and $\gcd(i,2r)=1$.
\end{enumerate}
A folk conjecture in the subject asserts that there are no further examples.  This is known to be true for $r=1$ \cite{J} and $r=2$ \cite{C},
and also for $r=4$ if $p>3$ \cite{CL}.  
However, as noted in \cite{CLrev}, the methods used in these papers will likely not extend to much larger values of $r$.
In this paper we prove this conjecture
for all large $r$:

\begin{thm}\label{main}
If $x^t$ is a planar function on $\F_{p^r}$, 
where $p^r\ge (t-1)^4$ and $p\nmid t$,
then either 
\emph{(1.1)} or \emph{(1.2)} holds.
\end{thm}

Note that each of the known planar monomials over $\F_q$ has the property that it is also planar over $\F_{q^k}$ for infinitely many integers $k$.  One consequence of our result is that no other planar monomials have this property:

\begin{cor} \label{maincor}
For any prime $p$ and any positive integer $t$, the function $c\mapsto c^t$ is a planar function on $\F_{p^k}$ for infinitely many $k$ if and only if either
\begin{itemize}
\item $t=p^i+p^j$ where $p$ is odd and $i\ge j\ge 0$; \,or
\item $t=\frac{3^i+3^j}2$ where $p=3$ and $i>j\ge 0$ with $i\not\equiv j\pmod{2}$.
\end{itemize}
\end{cor}

This corollary resolves two conjectures of Hernando, McGuire and Monserrat \cite[Conjectures PN2 and PN3]{HMM}.  It is the first known characterization of the known planar monomials among all planar monomials.

Theorem~\ref{main} (in a slightly weaker form) was proved in the case $t\equiv 1\pmod{p}$ by
Leducq \cite{L}.  Thus, the bulk of our effort addresses the case $t\not\equiv 1\pmod{p}$.
In this case, Theorem~\ref{main} (again in a slightly weaker form) was proved in \cite{HMM} if $t$ and $p$
satisfy any of eight different conditions.  We show that none of these extra conditions are needed.
Our approach is quite different from that of \cite{L} and \cite{HMM}.  Whereas those papers rely on dozens of pages of computations involving the singularities of an associated (possibly singular and reducible) plane curve, we focus on the functional decomposition of a certain univariate polynomial.  In particular, our most novel contribution is a method for testing whether a polynomial can be written as a function of a Dickson polynomial.

After reviewing some background material in the next section, we prove
Theorem~\ref{main} and Corollary~\ref{maincor} in Sections 3 and 4, respectively.  Then in Section 5 we show that our techniques yield a simple proof of the Segre--Bartocci conjecture about hyperovals in Desarguesian projective planes.


\section{Background results}

In this section we present the known results about exceptional polynomials, Dickson polynomials, and functional 
decomposition which will be used in our proofs.  We begin with some definitions.

\begin{defn}
A polynomial $F(x)\in\F_q[x]$ is \textit{linear} if it has degree one.
\end{defn}

\begin{remark}
What we call linear polynomials are sometimes called affine polynomials.
\end{remark}

\begin{defn}
A polynomial $F(x)\in\F_q[x]$ of degree at least $2$ is \textit{indecomposable} if there do not exist nonlinear
$G,H\in\F_q[x]$ such that $F(x)=G(H(x))$.
\end{defn}

\begin{defn}
A polynomial $F(x)\in\F_q[x]$ is \textit{exceptional} if there are infinitely many $k$ for which the
function $c\mapsto F(c)$ is a bijection on $\F_{q^k}$.
\end{defn}

Plainly every polynomial in $\F_q[x]$ of degree at least $2$ can be written as the composition of
indecomposable polynomials in $\F_q[x]$.  Moreover, for $G,H\in\F_q[x]$, 
if $G(H(x))$ is exceptional then both $G$ and $H$ are
exceptional (in fact the converse holds as well \cite{Z}, but it will not be used in this paper).
Thus, every nonlinear exceptional polynomial is the composition of indecomposable exceptional polynomials.
Much difficult mathematics has been used in the study of indecomposable exceptional polynomials (see e.g.\ \cite{GRZ,GZ,Z}),
and much remains to be done.  However, we will not need any deep results about exceptional polynomials.
Instead we will only rely on the following two known results.

\begin{prop} \label{weil}
If $F(x)\in\F_q[x]$ has degree at most $q^{1/4}$, and the function $c\mapsto F(c)$ induces a bijection on\/ $\F_q$, then $F$ is exceptional.
\end{prop}

\begin{prop} \label{exc}
If $F(x)\in\F_q[x]$ is an indecomposable exceptional polynomial of degree coprime to $q$, then there are linear
$\mu,\nu\in\F_q[x]$ such that $\mu\circ F\circ\nu$ is one of the following polynomials:
\begin{itemize}
\item $x^m$ for some prime $m$ which is coprime to $q-1$, or
\item $D_n(x,a)$ for some $a\in\F_q^*$ and some prime $n$ which is coprime to $q^2-1$.
\end{itemize}
\end{prop}

In this result, $D_n(x,a)$ denotes the degree-$n$ Dickson polynomial of the first kind with parameter $a$.
This is a polynomial in $\F_q[x]$ which satisfies the functional equation
\[
D_n(x+\frac{a}{x},a) = x^n + \left(\frac{a}{x}\right)^n.
\]
These polynomials are closely related to the Chebyshev polynomials of the first kind.
Here we note only that, in light of the above functional equation, $D_n(x,a)$ has degree $n$
and satisfies
$D_n(-x,a)=(-1)^n D_n(x,a)$.  Thus, 
if $n$ is odd then $D_n(x,a)$ is an odd polynomial,
in the sense that all of its terms have odd degree.
For more information about Dickson polynomials, see \cite{ACZ,LMT}.

\begin{remark}
Proposition~\ref{weil} follows easily from Weil's bound on the number of rational points on a curve over a finite field.  See \cite[Rem.~8.4.20]{Z} for the history of this result.  Proposition~\ref{exc} is a slight variant of a result from \cite{K}; see \cite{M} for a proof in the stated form.  The proof of this result
only depends on Weil's bound,
 group-theoretic results due to Burnside and Schur, and a quick and easy genus computation.
Weil's bound follows from the Riemann-Roch theorem, and the two group-theoretic results are
proved in a few pages in \cite{LMT}.  So, although deep tools have been used in the study of exceptional polynomials,
Propositions~\ref{weil} and \ref{exc} do not depend on such tools.
\end{remark}

The next result is well-known, but we include a proof for the reader's convenience.

\begin{lemma} \label{odd}
For $G,H\in\F_q[x]$, if $G\circ H$ is an odd polynomial and $\deg(G)$ is coprime to $q$ then $H(x)-H(0)$ is odd.
\end{lemma}

\begin{proof}
Suppose otherwise.  Let $ax^{\alpha}$ and $bx^{\beta}$ be the leading terms of $G$ and $H$, respectively,
and let $cx^{\gamma}$ be the highest-degree term of $H$ having even degree.  Then $\alpha$ and $\beta$ are
odd, and $\gamma$ is both even and positive.  
Writing $\delta:=(\alpha-1)\beta+\gamma$, and noting that $\delta$ is even, it
follows that the coefficient of $x^{\delta}$
in $G\circ H$ is $a\alpha b^{\alpha-1}c$.  Since this is nonzero, $G\circ H$ has a term of even degree,
which contradicts our hypothesis.
\end{proof}

\begin{remark}
The above lemma has been rediscovered many times.  It is not true without the hypothesis on $\deg(G)$; one
counterexample from \cite{BZ} is $G=(x+1)^s(x-1)^{q-s}$ and $H=x^q+(x+1)^{q-s}(x-1)^s$ with $q$ odd and
$0<s<q$.
\end{remark}

\begin{lemma} \label{linears}
Let $\mu,\nu\in\F_q[x]$ be linear, and let $G\in\F_q[x]$ have degree larger than $1$ and coprime to $q$.
If both $G$ and $\mu\circ G\circ\nu$ are odd, then $\nu(0)=0$.
\end{lemma}

\begin{proof}
Write $\nu(x)=cx+d$, and let $ax$ and $bx^{\beta}$ be the leading terms of $\mu$ and $G$.
Then the coefficient of $x^{\beta-1}$ in $\mu\circ G\circ\nu$ is $ab\beta c^{\beta-1}d$.
But this coefficient is zero by hypothesis, so $d=0$.
\end{proof}

Finally, we recall Lucas's theorem about binomial coefficients mod $p$
(see e.g. \cite{Lucas,Granville}):
\begin{lemma} \label{lucas}
Let $p$ be prime and let $m$ and $n$ be positive integers.  Write
$m=m_0+m_1 p + m_2 p^2 + \dots + m_s p^s$
and $n:=n_0+n_1 p+n_2 p^2 + \dots + n_s p^s$
where $0\le m_i,n_i\le p-1$ for each $i$.  Then
\[
\binom{n}{m} \equiv \binom{n_0}{m_0}
\binom{n_1}{m_1}\dots\binom{n_s}{m_s} \pmod{p}.
\]
\end{lemma}


\section{Proof of Theorem~\ref{main}}

We now prove Theorem~\ref{main}.
Suppose that $x^t$ is a planar function on $\F_{p^r}$ where $p^r\ge (t-1)^4$ and $t>2$
and $p\nmid t$.
Planarity implies that, for $\hat F(x):=(x+1)^t-x^t$, the function $c\mapsto \hat F(c)$ is a
bijection on $\F_{p^r}$.  By Proposition~\ref{weil}, $\hat F$ is an exceptional polynomial
over $\F_p$, so there are infinitely many $k$ for which $\hat F$ induces a bijection on $\F_{p^k}$;
equivalently, there are infinitely many $k$ for which $x^t$ is a planar function on $\F_{p^k}$.
If $t\equiv 1\pmod{p}$ then \cite[Cor.~1.7]{L} implies that (1.1) holds.  Henceforth assume that
$t\not\equiv 1\pmod{p}$.

Since $\hat F(-1-x)=(-1)^t \hat F(x)$, bijectivity of $\hat F$ on $\F_{p^r}$ implies that $p$ is odd.
Thus also $c\mapsto F(c)$ is a bijection on $\F_{p^r}$, where
\[
F(x) := 4^t\hat F\Bigl(\frac{x}4-\frac12\Bigr) = (x+2)^t-(x-2)^t.
\]
As above, Proposition~\ref{weil} implies that $F$ is an exceptional polynomial over $\F_p$.
Since $p\nmid t$, we have $\deg(F)=t-1$, so $\deg(F)>1$.   Hence $F$ can be written as the
composition of indecomposable polynomials over $\F_p$.  Write $F=F_1\circ F_2$ with $F_i\in\F_p[x]$
and $F_2$ indecomposable.  Since $F$ is exceptional, it follows that $F_2$ is exceptional.
Since $\deg(F)=t-1$ is coprime to $p$, by Proposition~\ref{exc} there are linear
$\mu,\nu\in\F_q[x]$ such that either
\begin{equation}
F_2=\mu\circ x^m\circ\nu \,\text{ for some prime $m$ which is coprime to $p-1$} \label{cyclic} 
\end{equation}
or
\begin{align}
F_2=&\mu\circ D_n(x,a)\circ\nu \,\text{ for some $a\in\F_p^*$ and some prime $n$} \label{dickson}\\
&\text{which is coprime to $p^2-1$.}  \notag
\end{align}
In particular, since $p$ is odd, it follows that in (\ref{cyclic}) we have $m\ge 3$ and $(m,2p)=1$,
and in (\ref{dickson}) we have $n\ge 5$ and $(n,6p)=1$.

Next, note that $F(-x) = (-1)^{t+1} F(x)$.
Since $F$ induces a bijection on $\F_{p^r}$, and $p>2$, we cannot have $F(-x)= F(x)$, so $t$ is even
and $F$ is an odd polynomial.  Now Lemma~\ref{odd} implies that $F_2(x)-F_2(0)$ is odd.
Since $F_2=\mu\circ H\circ\nu$ where $\mu,\nu\in\F_p[x]$ are linear and $H$ is either $x^m$ or $D_n(x,a)$,
and since further $\deg(H)$ is odd, we know that $H$ is odd.  By Lemma~\ref{linears},
we must have $\nu(x)=c_0 x$ with $c_0\in\F_p^*$.
Thus $F = G\circ H(c_0x)$ where $G:=F_1\circ\mu$ is in $\F_p[x]$.
If $H=x^m$ with $m\ge 3$ then $F\in \F_{p}[x^m]$, which is false since the coefficient of $x$
in $F$ is $t( 2^{t-1} - (-2)^{t-1} )$ which (since $t$ is even) equals
$t2^t$, and in particular is nonzero.
Thus we must have $H(x)=D_n(x,a)$ where $n\ge 5$ is odd and $a\in\F_p^*$.  Now put
\[
A(x):=F \circ \frac{1}{c_0}\left(\frac{x}{a}+\frac{a}{x}\right).
\]
Then
\[
A(x) = G\circ \left( \Bigl(\frac{x}{a}\Bigr)^n+\Bigl(\frac{a}{x}\Bigr)^{n}\right)
\]
is an element of $\F_{p}[x^n,x^{-n}]$.  But also
\[
A(x) = \Bigl( \frac{x}{c_0a}+\frac{a}{c_0x} + 2 \Bigr)^t - \Bigl( \frac{x}{c_0a}+\frac{a}{c_0x} - 2\Bigr)^t.
\]
Now put $B(x):=\frac12c_0^t A(ax)$ and $c:=2c_0$; then $B\in\F_{p}[x^n,x^{-n}]$ and
\[
B(x) = \frac12\Bigl( ( x+x^{-1}+ c )^t - ( x + x^{-1} - c )^t\Bigr).
\]
The coefficient of $x^{t-1}$ in $B(x)$ is $tc$, which is nonzero, so $n\mid (t-1)$.
Since $B$ is a Laurent polynomial all of whose terms have degree divisible by $n$, and $n\ge 5$ is odd, it follows that the coefficients of $x^{t-3}$, $x^{t-5}$, and $x^{t-7}$ in $B$ must be zero.  We now compute these coefficients.

The coefficient of $x^{t-3}$ in $B$ is  
\[
t c (t-1) + \binom{t}{3} c^3,
\]
which equals
\[
ct(t-1)\Bigl(\frac{t-2}6 c^2 + 1\Bigr).
\]
Since this coefficient is zero, we must have
\begin{equation} \label{t-3}
\frac{t-2}{6}c^2 = -1.
\end{equation}
Here, if $p=3$, we first interpret $\frac{t-2}{6}$ as a rational number, and then view this rational number as an element of $\F_p$; in particular, if $p=3$ then $t\equiv 2\pmod{3}$ but $t\not\equiv 2\pmod{9}$.

Suppose for the moment that neither of the following holds:
\begin{align}
p>3 \,\, &\text{ and }\,\, t\equiv \frac12 \,\text{ (mod~$p$), \,\,  or}  \label{1} \\
p=3 \,\, &\text{ and }\,\, t\equiv \frac12 \,\text{ (mod~$9$).} \label{2}
\end{align}
We will obtain a contradiction from the fact that the coefficients of $x^{t-5}$ and $x^{t-7}$ in
 $B$ are zero.
The coefficient of $x^{t-5}$ in $B$ is
\[
t c \binom{t-1}{2} + \binom{t}{3} c^3 (t-3) + \binom{t}{5} c^5;
\]
by using (\ref{t-3}), we can simplify this expression to
\[
c^3 \frac{t(t-1)}4 \frac{(t-\frac{1}{2})(t-4)}{15}.
\]
Since this equals zero, but neither (\ref{1}) nor (\ref{2}) holds, we must have either
\begin{align}
&\text{$p>5$ \,\, and \,\, $t\equiv 4$ (mod~$p$), \,\, or } \label{p>5} \\
&\text{$p=5$ \,\, and \,\, $t\equiv 4$ (mod~$25$).}  \label{p=5}
\end{align}
In particular, $p>3$, so (since (\ref{1}) does not hold) we have $t\not\equiv \frac12\pmod{p}$.
Next, the coefficient of $x^{t-7}$ in $B$ is
\[
t c \binom{t-1}{3} + \binom{t}{3} c^3 \binom{t-3}{2} + \binom{t}{5} c^5 (t-5)
  + \binom{t}7 c^7;
\]
again using (\ref{t-3}), we can simplify this expression to
\[
-c^5 t(t-1)\frac{(t+1)(t-\frac{1}{2})(t-3)(t-5)}{3^3\cdot 5\cdot 7}.
\]
Since this equals zero, and $t\equiv 4\not\equiv \frac12\pmod{p}$, we must have $t\equiv -1\pmod{p}$,
whence $p=5$.  But since $p=5$, the vanishing of the coefficient of $x^{t-7}$ implies that
$t\equiv -1\pmod{25}$, which contradicts (\ref{p=5}).  This contradiction shows that in fact either
(\ref{1}) or (\ref{2}) must hold.

Now assume that either (\ref{1}) or (\ref{2}) holds.
In either case, (\ref{t-3}) implies that $c^2=4$, so $c=2\epsilon$ with $\epsilon\in\{1,-1\}$.
Hence
\[
2B(x) = ( x+x^{-1} + 2\epsilon )^t - ( x+x^{-1} - 2\epsilon )^t.
\]
Since $B(x)\in\F_{p}[x^n,x^{-n}]$, it follows that $B(x^2)\in\F_{p}[x^{2n},x^{-2n}]$.
But we compute
\[
2B(x^2) = (x+\frac{\epsilon}x)^{2t} - (x-\frac{\epsilon}x)^{2t} = 2\sum_{\substack{0<i<2t \\ i \text{ odd}}} \binom{2t}{i} \epsilon^i x^{2t-2i}.
\]
Since $n$ is odd and $n\mid (t-1)$, we see that $n\mid (2t-2i)$ if and only if $n\mid (i-1)$.
Thus, the condition $B(x^2)\in\F_{p}[x^{2n},x^{-2n}]$ asserts that
\[
\text{if $i$ is odd and } \,i\not\equiv 1 \text{ (mod $n$)\, then } \,\binom{2t}{i}\equiv 0 \text{ (mod $p$)}.
\]
Write  $2t = \sum_{j=0}^s e_j p^j$  where $e_j$ is an integer satisfying  $0\le e_j\le p-1$.
Since either (\ref{1}) or (\ref{2}) holds, we have $2t\equiv 1\pmod{p}$, so $e_0=1$.
First consider $i=1+2p^j$ for any $1\le j\le s$.  Clearly $i$ is odd, and since the only prime factors of $i-1$ are $2$ and $p$, neither of which divides $n$, we also have $i\not\equiv 1\pmod{n}$.
Thus we must have $\binom{2t}{i}\equiv 0\pmod{p}$, so Lucas's theorem (Lemma~\ref{lucas}) implies $e_j<2$.  Hence every $e_j$ is either $0$ or $1$.  Next, for any $0<j<k$, consider the three values $i_1=p^j$, $i_2=p^k$, and
$i_3=1+p^j+p^k$.  Each of these values is odd, but since $i_3=1+i_1+i_2$ we cannot have
$i_1\equiv i_2\equiv i_3\equiv 1\pmod{n}$.  Thus there is some $i\in\{p^j,p^k,1+p^j+p^k\}$
for which $\binom{2t}{i}\equiv 0\pmod{p}$, so (again by Lucas's theorem) we must have either $e_j=0$ or $e_k=0$.  Since $2t>1$, the only remaining possibility is that $2t=1+p^s$ for some $s>0$.  
Writing $q:=p^s$, we compute
\begin{align*}
 F\circ (x^2+x^{-2}) &= (x^2+x^{-2}+2)^t - (x^2+x^{-2}-2)^2 \\
&= (x+x^{-1})^{1+q} - (x-x^{-1})^{1+q} \\
&= 2(x^{q-1}+x^{1-q}).
\end{align*}
But also $D(x):=D_{\frac{q-1}2}(x,1)$ satisfies
\[
D \circ (x^2+x^{-2}) = x^{q-1} + x^{1-q},
\]
so $F(x)=2D(x)$.
Since $F(x)$ is exceptional over $\F_p$, it follows that $D(x)$ is exceptional over $\F_p$ as well.
By an easy classical result (see e.g.\ \cite[Thm.~54]{D}), $D(x)$ is exceptional over $\F_p$ if and only if
$\frac{q-1}2$ is coprime to $p^2-1$.
Since both $\frac{q-1}2$ and $p^2-1$ are divisible by $\frac{p-1}2$, we must have $p=3$.
Finally, since $t=\frac{p^s+1}2$ is even, we see that $s$ is odd.
This concludes the proof.


\section{Proof of Corollary~\ref{maincor}}

We now prove Corollary~\ref{maincor}.
The ``if'' direction is known and easy: for, if $p$ is odd and $i\ge j\ge 0$ then
\[
(x+1)^{p^i+p^j} - x^{p^i+p^j} -1 = x^{p^i}+x^{p^j}
\]
induces a homomorphism from the additive group of $\F_{p^k}$ to itself, and therefore induces a bijection on $\F_{p^k}$ if and only if it contains no nonzero roots in $\F_{p^k}$.  But the nonzero elements of the kernel are the $(p^{i-j}-1)$-th roots of $-1$ in $\F_{p^k}$.  There are no such roots of $-1$ if $i=j$,
and if $i\ne j$ then there are no such roots in $\F_{p^k}$ when $(i-j)\mid k$.
Thus $x^{p^i+p^j}$ is planar on $\F_{p^k}$ for infinitely many $k$.
Next, for $t=\frac{3^i+3^j}2$ where $p=3$ and $i>j\ge 0$ and $i\not\equiv j\pmod{2}$,
the argument at the end of the previous section shows that
\[
(x+1)^t - x^t = - D_s(x-1,1)
\]
where $s:=\frac{3^i-3^j}2$.
Since $s$ is coprime to $3-1=2$, Dickson's result
\cite[Thm.~54]{D} implies that $D_s(x,1)$ is exceptional over $\F_3$,
so $x^t$ is planar on $\F_{3^k}$ for infinitely many $k$.

Conversely, fix a prime $p$ and a positive integer $s$, and suppose that $c\mapsto c^s$ is a planar function on $\F_{p^k}$ for infinitely many $k$.
Writing $s=p^j t$ with $p\nmid t$, it follows that $c\mapsto c^t$ is planar on $\F_{p^k}$ for infinitely many $k$.  In particular, $c\mapsto c^t$ is planar on $\F_{p^r}$ for some $r$ such that $p^r\ge (t-1)^4$, so 
Theorem~\ref{main} implies that either (1.1) or (1.2) holds.
The result follows.


\section{The Segre--Bartocci conjecture}

We now show how a simple modification of our argument yields a new proof of the Segre--Bartocci conjecture about monomial hyperovals in finite Desarguesian projective planes.  
This conjecture was only proved for the first time quite recently, by Hernando and McGuire \cite{HM}, by means of a lengthy calculation involving singularities of a certain plane curve.  Our proof is considerably shorter and simpler.

A \textit{hyperoval} in $\mathbb{P}^2(\F_q)$ is a set of $q+2$ points such that no three are collinear.  It turns out that such objects can only exist if $q$ is even.  It is easy to see that, for a suitable choice of coordinates on $\mathbb{P}^2(\F_q)$, any hyperoval can be written in the form
\[
\{(1:c:H(c))\colon c\in\F_q\}\cup\{(0:0:1), (0:1:0)\}
\]
for some $H(x)\in\F_q[x]$.  Denote this set by $D(H(x))$.
Segre and Bartocci conjectured in 1971 \cite{SB} that the
values $t=6$ and $t=2^i$ (with $i>0$) are the only positive integers $t$
for which there are infinitely many $k$ such that $D(x^t)$ is a hyperoval
in $\mathbb{P}^2(\F_{2^k})$.  By considering slopes of lines between two points, one can reformulate this conjecture as asserting that $t=6$ and $t=2^i$
are the only positive integers $t$ for which the polynomial
$\hat F(x):=x^{t-1}+x^{t-2}+\dots+1$ is exceptional over $\F_2$ (see, for instance,
\cite[p.~505]{LN}).  Assume $\hat F$ is exceptional.  Then $\hat F(0)\ne \hat F(1)$, so $t$ is even.  Assume $t>2$, so that $\deg(F)=t-1>1$.  Put $F(x):=\hat F(x+1)$, so
\[
F(x) = \frac{(x+1)^t+1}x.
\]
Then $F$ is an odd polynomial, so the first part of the proof of Theorem~\ref{main}
shows that $F=G\circ H$ where $H$ is either $x^m$ (with $m>1$ odd) or
$D_n(x,1)$ (with $n>1$ coprime to $6$).
If $F=G(x^m)$ with $m>1$ odd then Lucas's theorem implies that $t$ is a power of $2$: for, if $2^j$ and $2^k$ are distinct terms in the binary expansion of $t$, then $F$ has terms of degrees $2^j-1$ and $2^k-1$ and $2^j+2^k-1$, and the gcd of these three degrees is $1$.
Finally, suppose $F=G\circ D_n(x,1)$.  Then we have
$F(x+x^{-1})=G(x^n+x^{-n})\in \F_2[x^n,x^{-n}]$.
In order to compute the coefficients of the Laurent polynomial $F(x+x^{-1})$,
we write
\[
F(x+x^{-1}) = \sum_{i=1}^t \binom{t}{i} (x+x^{-1}+1)^{i-1}.
\]
Since the coefficient of $x^{t-1}$ in $F(x+x^{-1})$ is nonzero, we see that $n\mid (t-1)$,
so that the coefficients of $x^{t-3}$ and $x^{t-7}$ must be zero.
But one easily checks that this only occurs when $t=6$, which implies the
Segre--Bartocci conjecture.

\end{document}